\documentclass[a4paper,12pt,oneside]{article}

\usepackage{mathrsfs}
\usepackage{mathtools}
\usepackage{tikz}
\usetikzlibrary{shapes,arrows,positioning} \usepackage[framemethod=tikz]{mdframed}
\usepackage{soul}
\usepackage{transparent}
\usetikzlibrary{arrows, decorations.markings}
\usetikzlibrary{automata,positioning}
\usepackage[T1]{fontenc}
\usetikzlibrary{shadings,shadows,shapes.arrows}
\usetikzlibrary{shapes,snakes}
\usetikzlibrary{matrix}
\usepackage{fancybox}
\usepackage{etex}
\usepackage{tcolorbox}
\usepackage{color}
\usepackage{fancybox,framed}

\usepackage[framemethod=tikz]{mdframed}
\usepackage{lipsum}
\usepackage{amssymb}
\usepackage{pifont}
\definecolor{mycolor}{rgb}{0.122, 0.435, 0.698}
\usepackage{emptypage}
\usepackage{afterpage}

\usepackage{epigraph}
\usepackage{url}
\usepackage{fancyhdr}
\usepackage{enumitem}
\usepackage{wrapfig}
\usepackage[toc,page]{appendix}
\usepackage{fp}
\usepackage{pgf}
\usepackage{xifthen}

\usepackage[utf8]{inputenc}

\usepackage{color}
\usepackage{import}
\usepackage{setspace}
\usepackage{latexsym,amsmath,amsfonts,amscd,amssymb,amsthm}
\usepackage{graphicx, caption}
\usepackage[all]{xy}
\usepackage{epstopdf}
\DeclareGraphicsExtensions{.pdf,.png,.jpg,.gif, .eps}

\usepackage{amsthm}
\usepackage{thmtools}

\theoremstyle{plain} 
\newtheorem{teo}{Theorem}
\newtheorem{lem}{Lemma}
\newtheorem{prop}{Proposition}

\newtheorem{preg}{Question}

\newtheorem{refi}{Refinement}
\theoremstyle{definition}

\theoremstyle{remark}
\newtheorem{obs}{Remark}
\newtheorem{ej}{Example}

\newcommand{\mapeo}[5]{
	\begin{eqnarray*}
		#1:#2 & \longrightarrow & #3\\
		#4 & \longmapsto & #5
\end{eqnarray*}}

\newcommand{\U}[1]{U_{2\varepsilon_{#1}}(A_{#1})}
\newcommand{\todon}{n\in\mathbb{N}}

\renewcommand{\epsilon}{\varepsilon}
\newcommand{\ball}[2]{\textrm{B}(#1,#2)}
\newcommand{\dist}[2]{\textrm{d}(#1,#2)}
\newcommand{\fas}{\{\varepsilon_n,A_n\}_{\todon}}

\newcommand{\diam}{\textrm{diam}}

\newcommand{\card}{\textrm{card}}

\usepackage{hyperref}
\hypersetup{
	colorlinks=true,
	linkcolor=blue,
	filecolor=magenta,      
	urlcolor=cyan,
	pdfpagemode=FullScreen,
}

\usepackage[light, math]{iwona}
\usepackage[T1]{fontenc}
\makeindex
\makeatletter
\newcommand{\subjclass}[2][2010]{%
	\let\@oldtitle\@title%
	\gdef\@title{\@oldtitle\footnotetext{#1 \emph{Mathematics Subject Classification.} #2}}%
}
\newcommand{\keywords}[1]{%
	\let\@@oldtitle\@title%
	\gdef\@title{\@@oldtitle\footnotetext{\emph{Key words and phrases.} #1.}}%
}
\makeatother
\title{\bf{Finite approximations of countable metric and ultrametric compacta}}
\usepackage{authblk}
\author{Diego Mondéjar}
\date{}                     

\begin{document}
	\maketitle
	\begin{abstract}
Adapting a homotopy reconstruction theorem for general metric compacta, we show that every countable metric or ultrametric compact space can be topologically reconstructed as the inverse limit of a sequence of finite $T_0$ spaces which are finer approximations of the space.
	\end{abstract} 
\section{Introduction}
The realization or approximation of compact metric spaces (and more general topological spaces) using inverse limits is a recurrent and significant theme in topology. For example, it is very useful for the handling of some attractors in dynamical systems, as solenoid spaces. See the beautiful surveys \cite{Kattractors, SGdynamical}. There are results back from Alexandroff \cite{Adiskrete} in this direction and, since then, a lot of results have been achieved. One of the fields where this problem is addressed is shape theory: roughly, the idea is to define spaces as inverse limits of inverse systems of simpler ones, so we can define maps between the spaces as maps between the systems for an easier treatment. For spaces with a "good" behaviour (\textsc{ANR}s), this is nothing but the homotopy type. But, for spaces with bad local properties, this enlarge the set of morphisms so we can still detect interesting topological properties, called shape properties. Shape theory was initiated by Borsuk \cite{Bconcerning} and subsequently developed by many authors. For a comprehensive treatment, we recommend the book \cite{MSshape} by Mardešić and Segal.

Our approach is closely related with shape theory. We use inverse sequences of finite topological spaces. The idea of computing topological properties from finite approximations of the space comes from Robins \cite{Rtowards}, where she introduced the concept of persistence of Betti numbers using inverse systems of finite approximations of the space in study and propose shape theory as a tool to perform this computation. In the last years, there has been an increasing interest in the use of finite topological spaces, related with the development of Computational and Applied Topology \cite{Ctopology, Gbarcodes}. A lot of work in finite topological spaces has been done, and its interest is still growing (see, for instance, \cite{Mfinitetopological, Mfinitecomplexes, Balgebraic}). There are recent reconstructions of topological properties of spaces by inverse systems or sequences of finite topological spaces. For instance, in \cite{Cinverse} it is shown that every compact polyhedron has the same homotopy type that an inverse limit of finite $T_0$ spaces containing a homeomorphic copy of the polyhedron. This result is generalized later in \cite{MMreconstruction} for compact metric spaces and in \cite{Bilskiinverse} for locally compact, paracompact Hausdorff spaces. In \cite{KWfinite, KTWthe}, every compact Hausdorff space is reconstructed as the Hausdorff reflection of an inverse limit of finite $T_0$ spaces. 

This paper is a continuation of the works \cite{MMreconstruction} and \cite{Mpolyhedral}, in which it is shown that for the general class of metric compacta, we can recover the homotopy type and the shape type as inverse limits of the finite approximations and of some polyhedra associated with them, respectively. The aim of this paper is to show that compact metric spaces which are countable or ultrametric can be completely reconstructed as inverse limits of finite $T_0$ spaces, which are approximations of the space considered. The importance of these reconstructions if that they are explicitly calculated and can be described and programmed for concrete examples, making them usable for practical computational goals as in \cite{MMreconstruction} and \cite{Chocomp}.

The article is structured in the following way. First, in the next section, we briefly retake the construction used in  \cite{MMreconstruction} and \cite{Mpolyhedral}, also using a detailed example for a simple space, the unit interval. Then, in Section \ref{sec:refine}, we consider some refinements of the construction to enhance its reconstruction capabilities. This machinery is used in Section \ref{result} to prove the main results, namely, that every countable metric or ultrametric compact space is homeomorphic to an inverse limit of finite $T_0$ spaces.

\section{Finite approximative sequences for compact metric spaces}
Here, we review the Main Construction for a compact metric space $X$ introduced in \cite{MCLepsilon} and sharpened in \cite{MMreconstruction} to obtain an inverse sequence of finite $T_0$ spaces that reconstructs homotopically $X$. Recall that finite topological spaces has a minimal base consisting of its minimal neighborhoods, that is, the intersection of every open set of each point. Finite $T_0$ spaces are in bijective correspondence with \textsc{poset}s and that continuous maps between them are just order preserving functions \cite{Mfinitetopological}. For the sake of completeness, we add here the definition of inverse sequences and limits. An \emph{inverse sequence} $\left(X_n,p_{n,n+1}\right)_{n\in\mathbb{N}}$ of topological spaces is a countable set of spaces $\left(X_n\right)_{n\in\mathbb{N}}$ and continuous maps, called \emph{bonding maps}, $p_{n,n+1}:X_{n+1}\rightarrow X_{n}$ for $n\in\mathbb{N}$. For every $n<m$, we will write $p_{n,m}:X_{m}\rightarrow X_{n}$ as the composition $p_{n,m}=p_{n,n+1}\circ p_{n+1,n+2}\circ\ldots\circ p_{m-1,m} $. The \emph{inverse limit} $\mathcal{X}=\varprojlim\left(X_n,p_{n,n+1}\right)_{n\in\mathbb{N}}$ of an inverse sequence $\left(X_n,p_{n,n+1}\right)_{n\in\mathbb{N}}$ is the subset of the product space $\prod_{n\in\mathbb{N}}X_n\supset\mathcal{X}$ consisting of the points $(x_1,x_2,\ldots,x_n,x_{n+1},\ldots)$ satisfying $p_{n,n+1}(x_{n+1})=x_n$ for every $n\in\mathbb{N}$. Note that open subsets in $\mathcal{X}$ are of the form $\prod_{n\in\mathbb{N}}U_n\cap\mathcal{X}$, where $U_n$ is open in $X_n$, only different from $X_n$ for a finite number of indexes.

The \emph{Main Construction} over $X$ states that there exists a decreasing sequence of real numbers $\left\lbrace\varepsilon_n\right\rbrace_{\todon}$ tending to zero and \emph{adjusted} to a sequence of discrete subsets $\left\lbrace A_n\right\rbrace_{\todon}$ of $X$, that is, for every $n\in\mathbb{N}$, $A_n$ is a finite $\varepsilon_n$ approximation of $X$ (for every $x\in X$ there is a point of $A_n$ closer than $\varepsilon_n$) satisfying  $\varepsilon_{n+1}<\frac{\varepsilon_n-\gamma_{n}}{2},$ where $\gamma_{n}=\sup\left\lbrace\dist{x}{A_n}:x\in X\right\rbrace<\varepsilon_n.$ The Main Construction can be performed in every compact metric space $X$ (Theorem 3 in \cite{MMreconstruction}). Now, we can define the spaces $$U_{2\varepsilon_n}(A_n)=\left\lbrace C\in 2^{A_n}_u: \diam(C)<2\varepsilon_n\right\rbrace,$$ where $2^Z_u$ for a topological space $Z$ denotes the \emph{hyperspace} of non-empty closed sets of $Z$ with the \emph{upper semifinite} topology (see \cite{MGupper,MGhomotopical}) given by the base $\left\lbrace B(U)\right\rbrace$, for every $U$ open set in $X$ and being $B(U)$ the set of elements $C$ of $2^Z$ that are contained in $U$. It turns out that this topology, when $Z$ is a discrete set, which is the case, is just a finite $T_0$ space, that is, a \textsc{poset}, given by the relation 
$C\leqslant D \Longleftrightarrow C\subseteq D$.
Hence, our spaces $U_{\varepsilon_n}(A_n)$ are finite $T_0$ spaces with its topology given by this relation and hence the minimal neighborhood for a point $C\in U_{2\varepsilon_n}(A_n)$ is the set $2^C$. Consider the map that connects two consecutive spaces \mapeo{p_{n,n+1}}{U_{2\varepsilon_{n+1}}(A_{n+1})}{U_{2\varepsilon_n}(A_n)}{C}{\bigcup_{c\in C}\left\lbrace a\in A_n:\dist{a}{c}=\dist{A_n}{c}\right\rbrace,} for every $n\in\mathbb{N}$. that is, the closest points in the next approximation. A \emph{\textsc{fas}} is an inverse sequence $\left(U_{2\varepsilon_n},p_{n,n+1}\right)_{n\in\mathbb{N}}$ obtained in this way by the Main Construction. With this term, we will refer interchangeably to the sequences of numbers and approximations $\left(\varepsilon_n, A_n\right)_{n\in\mathbb{N}}$, as well as to the inverse sequence itself, since the former completely determine the latter. Then, using associated inverse sequences of polyhedra we can reconstruct the shape type of $X$ \cite{Mpolyhedral}. Moreover, its homotopy type can be obtained as the inverse limit of the sequence, as shown in the following.
\begin{teo}[Theorem 4 of \cite{MMreconstruction}]
Let $X$ be a compact metric space. The inverse limit of every \textsc{fas} $$\mathcal{X}=\varprojlim\left(\U{n},p_{n,n+1}\right)_{n\in\mathbb{N}}$$ contains a subspace $\mathcal{X}^*\subset\mathcal{X}$ homeomorphic to $X$ which is a strong deformation retract of $\mathcal{X}$.
\end{teo}
We briefly review the proof of this theorem, for the sake of a better understanding of what follows.
\begin{proof}[Scketch of the proof]
The proof of the theorem is to construct natural maps from the inverse limit to the original space and backguard and then check that the compositions are homotopic to the identity.

Every element of the inverse limit is of the following form $$\left( C_n\right)_{n\in\mathbb{N}}
\Longleftrightarrow\left\lbrace
\begin{array}{ll}
C_n\in U_{2\varepsilon_n}(A_n)\\
p_{n,n+1}(C_{n+1})=C_n
\end{array}\right.\forall n\in\mathbb{N}.$$
It is shown in Proposition 3 of \cite{MMreconstruction} that ${\left(C_n\right)_{n\in\mathbb{N}}}$ is a sequence of points in the Hausdorff hyperspace $2^X_H$ converging --in the Hausdorff metric-- to a singleton $\{x\}\in 2^X_H$. See Figure \ref{convergence}. 
\begin{figure}[h!]
	\captionsetup{width=.8\linewidth}
	\centering
	\begin{minipage}{0.55\textwidth}
		$$\left(C_n\right)_{n\in\mathbb{N}}=\left(\includegraphics[scale=0.35]{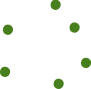}\enspace,\enspace\includegraphics[scale=0.35]{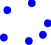}\enspace,\enspace\includegraphics[scale=0.35]{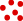}\enspace,\ldots\right)$$
	\end{minipage}
	\hfil
	\begin{minipage}{0.4\textwidth}
		\includegraphics[scale=0.45]{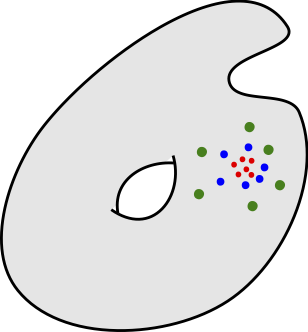}
	\end{minipage}
	\caption{An element $\left(C_n\right)_{n\in\mathbb{N}}$ of the inverse limit $\mathcal{X}$. Left: the sequence of points $\left(C_n\right)_{n\in\mathbb{N}}$ of the inverse limit $\mathcal{X}$. Right: the convergence of $\left(C_n\right)_{n\in\mathbb{N}}$ to a singleton $\{x\}\in 2^X_H$, that is, to a point $x\in X$ of the original space.}
	\label{convergence}
\end{figure}
Thus, it is natural to define a map from the inverse limit to the space as \mapeo{\varphi}{\mathcal{X}}{X}{\left(C_n\right)_{n\in\mathbb{N}}}{x,\enspace\text{with}\enspace\{x\}=\lim_H \left(C_n\right)_{n\in\mathbb{N}}.}For every $x\in X$ and $n\in\mathbb{N}$, we define the following sets: the first is formed by the elements in the approximation $A_n$ at distance smaller than $\varepsilon_n$ from $x$ and the second is the subset consisting of the intersection of all the images by the maps $p_{n,n+1}$ of this sets for every larger $n$:
$$
X^n:=A_n\cap\textsf{B}(x,\varepsilon_n)\enspace\enspace\text{and}\enspace\enspace
X_n^*:=\bigcap_{n<m} p_{n,m}(X^m)
$$
In this way, we consider all the candidates to be in the inverse limits and then select the ones that really make it true. This is the key idea of the proof. See Figure \ref{fig:connection}. 
\begin{figure}[h!]
	\captionsetup{width=.8\linewidth}
	\centering
	\includegraphics[width=0.5\linewidth]{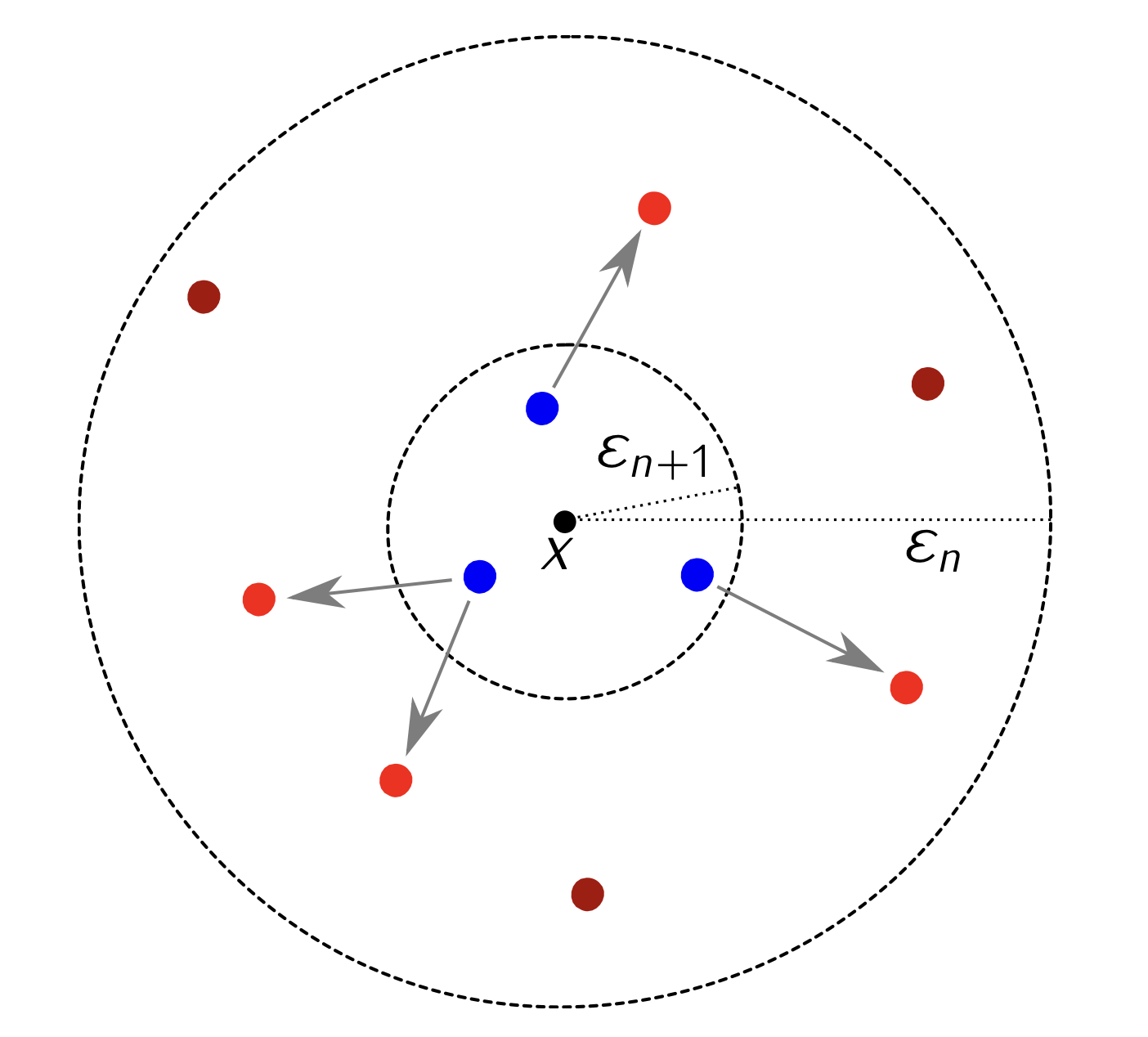}
	\caption{Connection of two consecutive stages. From the possible seven candidates of points in $X^n$ being in $X_n^*$, only four are reached as images of $p_{n,n+1}$.}
	\label{fig:connection}
\end{figure}
Then, we have that the sequence formed by the latter is an element of the inverse limit $\left(X_n^*\right)_{n\in\mathbb{N}}\in\mathcal{X}$. Also, it is shown that this elements are, in fact, the union of all the terms converging to the same $x$,
$$X_n^*=\bigcup_{\varphi(\{C_n\}_{n\in\mathbb{N}})=x}C_n,$$
which allow us to define the backguard map as \mapeo{\phi}{X}{\mathcal{X}}{x}{\left(X_n^*\right)_{n\in\mathbb{N}}}

This way, it is shown that $\varphi$ is surjective and $\phi$ is injective. Indeed, the subset $\mathcal{X}^*:=\phi(X)\subset\mathcal{X}$ is Hausdorff, and hence $\phi:X\rightarrow\mathcal{X}^*$ is a homeomorphism.  The two defined maps are homotopical inverses: $\varphi\cdot\phi:X\rightarrow X=1_X$ and $\phi\cdot\varphi:\mathcal{X}\rightarrow\mathcal{X}\simeq 1_{\mathcal{X}}$, with $H:\mathcal{X}\times[0,1]\rightarrow\mathcal{X}$ given by
\begin{equation*}
	H(\left(C_n\right)_{n\in\mathbb{N}},t) = \left\{
	\begin{array}{lll}
		\left(C_n\right)_{n\in\mathbb{N}} & \text{if} & t\in[0,1),\\
		\phi\cdot\varphi(\left(C_n\right)_{n\in\mathbb{N}}) & \text{if} & t=1.
	\end{array}\right.
\end{equation*}
\end{proof}
A simple example is very instructive for understanding the construction of these maps and their posterior adaptation to countable and ultrametric cases.
\begin{ej}\label{ex:interval}
We recall the example of the unit interval in \cite{MMreconstruction}. There, a \textsc{fas} for $I=[0,1]$ is constructed, by taking subdivisions of the interval of size some powers of $\frac{1}{3}$. Let us consider, $\varepsilon_1=2$ and $A_1=\{0\}$, we have $\gamma_1=1$. For $n>2$, consider $$\varepsilon_n=\frac{1}{3^{2n-3}}\enspace\text{and}\enspace A_n=\left\lbrace\frac{k}{3^{2n-3}}:k=0,1,\ldots,3^{2n-3}\right\rbrace.$$ The first two steps are depicted in Figure \ref{fig:interval}.
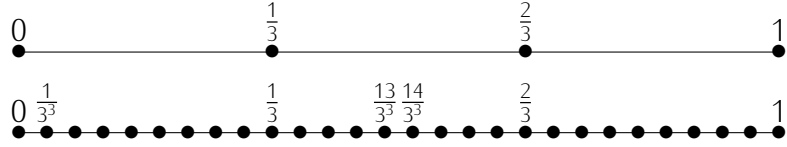
\begin{figure}[h!]
\captionsetup{width=.8\linewidth}
\centering
\begin{tikzpicture}[scale=10]
	\draw (0,0) -- (1,0);
	\draw [above] node at (0,0){0};
	\draw [above] node at (1,0){1};
	\draw [above] node at (0.33333333333,0){$\frac{1}{3}$};
	\draw [above] node at (0.66666666666,0){$\frac{2}{3}$};
	\foreach \x in {0,1/3,2/3,1} {\draw node at (\x,0) {$\bullet$} ;}
\end{tikzpicture}
\\
\begin{tikzpicture}[scale=10]
	\draw (0,0) -- (1,0);
	\draw [above] node at (0,0){0};
	\draw [above] node at (1,0){1};
	\draw [above] node at (0.33333333333,0){$\frac{1}{3}$};
	\draw [above] node at (0.66666666666,0){$\frac{2}{3}$};
	\draw [above] node at (0.03703703703,0){$\frac{1}{3^3}$};
	\draw [above] node at (0.48148148139,0){$\frac{13}{3^3}$};
	\draw [above] node at (0.51851851842,0){$\frac{14}{3^3}$};
	\foreach \x in {0,0.03703703703,...,1} {\draw node at (\x,0) {$\bullet$} ;}
\end{tikzpicture}
\caption{First two steps of the reconstruction of the unit interval. Note that every step is a refinement of the previous one.}
\label{fig:interval}
\end{figure}
Now we observe the elements of the inverse limit. Note that, in this example, $A_{n+1}$ is a subset of $A_n$ for all $\todon$. The image of a point $x\in I$ by the map $\phi$ can be of three types. The first, $x=\frac{k}{3^{2n_0-3}}$ is a point that is in the $\varepsilon_n$ approximations $A_n$ for every $n>n_0$, then $$\phi(x)=\left(\left\lbrace0\right\rbrace,p_{2,n}(\{x\}),\ldots,\overset{n_0}{\overset{\downarrow}{\{x\}}},\{x\},\ldots\right).$$ For instance, $0\in A_{n}$ for all $\todon$ and we have $\phi(0)=\left(0,0,\ldots\right)$. For points that are not in any of the approximations, we have a different structure. If a point is not special for these approximations, meaning it is not exactly in the middle of a pair of points, as $\frac{1}{5}$, then its image under $\phi$ is just the corresponding sequence of its expansion in base three, $$\phi\left(\frac{1}{5}\right)=\left(\left\lbrace0\right\rbrace,\left\lbrace\frac{1}{3}\right\rbrace,\ldots,\left\lbrace\frac{m}{3^{2n-3}}\right\rbrace,\ldots\right),$$ for some $m\in\mathbb{N}$ making $\frac{m}{3^{2n-3}}$ as close as possible to $\frac{1}{5}$. On the contrary, $\frac{1}{2}\notin A_n$ for any $\todon$, and it is exactly in the middle of two points of the approximations for every $n>1$, so we have $$\phi\left(\frac{1}{2}\right)=\left(\left\lbrace0\right\rbrace,\left\lbrace\frac{1}{3},\frac{2}{3}\right\rbrace,\left\lbrace\frac{14}{3^3},\frac{15}{3^3}\right\rbrace,\ldots,\left\lbrace\frac{\frac{3^{2n-3}-1}{2}}{3^{2n-3}},\frac{\frac{3^{2n-3}+1}{2}}{3^{2n-3}}\right\rbrace,\ldots\right).$$  This is the maximal element of $\mathcal{X}$ converging to $\frac{1}{2}$. But, two more elements do, namely 
\begin{align*}
C_1&=\left(\left\lbrace0\right\rbrace,\left\lbrace\frac{1}{3}\right\rbrace,\left\lbrace\frac{14}{3^3}\right\rbrace,\ldots,\left\lbrace\frac{\frac{3^{2n-3}-1}{2}}{3^{2n-3}}\right\rbrace,\ldots\right)\enspace\text{and}\nonumber\\
C_2&=\left(\left\lbrace0\right\rbrace,\left\lbrace\frac{2}{3}\right\rbrace,\left\lbrace\frac{15}{3^3}\right\rbrace,\ldots,\left\lbrace\frac{\frac{3^{2n-3}+1}{2}}{3^{2n-3}}\right\rbrace,\ldots\right).\nonumber
\end{align*}
So, there are three elements converging to $\frac{1}{2}$, thus making the map $\varphi$ not injective. 
\end{ej}

This failure in the injectivity of $\varphi$ is what causes that the inverse limit is homotopic but not homeomorphic to the original space. In general, the approximation of the \textsc{fas} is at homotopical level. The inverse limit $\mathcal{X}$ contains a homeomorphic copy $\phi(X)=\mathcal{X}^*$ of $X$ to which it strong deformation retract. But we can remake or tansform these approximations in order to make them more suitable for concrete topological spaces.

\section{Refining the Main Construction}{\label{sec:refine}}

In this section, we consider some refinements in the construction of our \textsc{fas} for compact metric spaces. Some of them will be used later to achieve complete reconstructions for countable spaces. Ultrametric spaces do not need these refinements to be completely reconstructed.

Let us generalize what we saw in Example \ref{ex:interval}: the map $\varphi$ is injective for every point that is in all approximations since some stage of the construction.
\begin{lem}\label{lem:inyectiva}
	Let $X$ be a compact metric space and suppose we perform the Main Construction obtaining a \textsc{fas} $\fas$. If $x\in X$ satisfies that there exists an $n_0\in\mathbb{N}$, with $x\in A_n$ for every $n\geqslant n_0$, then  $X_n^*=\{x\}$ for every $n\geqslant n_0$ and hence the cardinality of $\varphi^{-1}(x)$ is one, being $\varphi^{-1}(x)=X^*$.
\end{lem}
\begin{proof}
	Let us suppose we obtain the sequences $\{\varepsilon_n,A_n,\gamma_n\}_{\todon}$ performing the Main Construction over $X$. We are going to prove that, if $x\in X$ belongs to $A_n$ for every $n\geqslant n_0$, then $$X^*=\left(p_{1,n_0}(\{x\}),\ldots,p_{n_0-1,n_0}(\{x\}),\overset{n_0}{\overset{\downarrow}{\{x\}}},\{x\},\ldots\right).$$ So, if this is true, there are no more points $C\in\mathcal{X}$ satisfying $\varphi(C)=x$, apart from $X^*$ (because of the maximality of $X^*$). Let us prove it. For $n\geqslant n_0$, we have that $x\in X_n=\textrm{B}(x,\varepsilon_n)\cap A_n$, and then, $x\in X_n^*$ for every $n\geqslant n_0$, because $p_{n,m}(\{x\})=\{x\}$ for every $n_0\leqslant n<m$. So $X^*$ has the form $$X^*=\left(p_{1,n_0}(X^*_{n_0}),\ldots,p_{n_0-1,n_0}(X^*_{n_0}),X^*_{n_0},X^*_{n_0+1},\ldots\right).$$ Now we prove that $X^*_n=\{x\}$ for every $n\geqslant n_0$. Consider $y_0\in X_{n_0}$ different from $x$. Then, $y\in X^*_{n_0}$ if and only if there exists, for every $i\in\mathbb{N}$, $y_i\in A_{n_0+i}$ such that $y\in p_{n_0+i,n_0+i+1}(y_{i+1})$ and $y_i\in X_{n_0+i}$ for every $i\in\mathbb{N}$. We are going to see that, if there is a chain of points satisfiying the first condition, they cannot satisfy the second. So, let us suppose there exists a chain $y_i\in A_{n_0+i}$, for every $i\in\mathbb{N}$ such that one belongs to the image of the following. For the sake of simplicity, let us write $d_i:=\textrm{d}(x,y_i)$ for $i\in\mathbb{N}$, (and $d_0=\textrm{d}(x,y_0)$). For every $i\in\mathbb{N}$, $y_{i+1}$ is closer (or at the same distance) to $y_i$ than to $x$, so we have $$d_{i+1}\geqslant\textrm{d}(y_i,y_{i+1})<\gamma_{n_0+i}.$$ Moreover, it is obvious that $d_i\leqslant d_{i+1}+\textrm{d}(y_{i+1},y_i)$, i.e., $d_i-d_{i+1}\leqslant\textrm{d}(y_{i+1},y_i)$. Combining this with the previous observation, we get $d_i-d_{i+1}<\gamma_{n_0+1}$. On the other hand, we have that, for every $i\in\mathbb{N}$, $d_i\leqslant d_{i+1}+\textrm{d}(y_{i+1},y_i)\leqslant 2d_{i+1}$, so $d_{i+m}\geqslant\frac{d_i}{2^m}$. We supposed $y_0\in X_{n_0}$, so $\varepsilon_{n_0}-d_0>0$. We claim that, for every $i\in\mathbb{N}$, $$\varepsilon_{n_0+i}-d_i<\frac{2\varepsilon_{n_0}-(i+2)d_0}{2^{i+1}}.$$ We prove it by induction. The first case is
	\begin{eqnarray*}
		\varepsilon_{n_0+1}-d_1 & < & \frac{\varepsilon_{n_0}-\gamma_{n_0}}{2}-d_1<\frac{\varepsilon_{n_0}-d_0}{2}-\frac{d_1}{2}\\
		& \leqslant & \frac{\varepsilon_{n_0}-d_0}{2}-\frac{d_0}{2^2}=\frac{2\varepsilon_{n_0}-3d_0}{2^2}.
	\end{eqnarray*}
	Now, suppose the hypothesis of induction is satisfied, and we check
	\begin{eqnarray*}
		\varepsilon_{n_0+i+1}-d_{i+1} & < & \frac{\varepsilon_{n_0+i}-\gamma_{n_0+i}}{2}-d_{i+1}<\frac{\varepsilon_{n_0+i}-d_{i}}{2}-\frac{d_{i+1}}{2}\\
		& < & \frac{2\varepsilon_{n_0}-(i+2)d_0}{2^{i+2}}-\frac{d_0}{2^{i+2}}=\frac{2\varepsilon_{n_0}-(i+3)d_0}{2^{i+2}}.
	\end{eqnarray*}
	It is obvious that there exists an $i\in\mathbb{N}$ such that $(i+3)d_0>2\varepsilon_{n_0}$. For this $i$, we have that $\varepsilon_{n_0}-d_i<0$, so $y_i\notin X_{n_0+i}$, and then, $y_0\notin X^*_{n_0}$. We conclude $X^*_{n_0}=\{x\}$ and the same argument can be applied to show that $X_n^*=\{x\}$, for every $n\geqslant n_0$
\end{proof}

In view of Lemma \ref{lem:inyectiva}, it is natural to look for \textsc{fas} making the map $\varphi$ the "more injective" possible, i.e., injective in the largest possible set of points. The first observation we can do is that the points in the approximations are the better candidates for that. Given a compact metric space $X$, a \textsc{fas} satisfying $A_n\subset A_{n+1}$ for every $\todon$ will be called \emph{nested}.

\begin{refi}[Existence of nested \textsc{fas}]
Every compact metric space has a \emph{nested} \textsc{fas}.
\end{refi}
\begin{proof}[Construction]
Let us consider $\epsilon_1>0$ any real number and $A_1$ any $\varepsilon_1$ approximation. Then, consider $\varepsilon_2$ adjusted to $\varepsilon_1$ and for $A_2$ take the union $A_2=A'_2\cup A_1$ where $A'_2$ is a $\epsilon_2$ approximation of $X$, then so is $A_2$. We can proceed in this way for every $\todon$. If we have that $A_n$ is a $\varepsilon_n$ approximation of $X$, consider $\epsilon_{n+1}$ adjusted to them. Then consider $A_{n+1}$ as the union $A'_{n+1}\cup A_n$ where $A'_{n+1}$ is a $\epsilon_{n+1}$ approximation of $X$ and, hence, $A_{n+1}$ too. In this way, we obtain the desired nested \textsc{fas} of $X$.
\end{proof}

\begin{refi}[Nesting preserves adjustment]\label{refi:nesting}
For every compact metric space $X$ and a \textsc{fas} $\left(\varepsilon_n, A_n\right)_{n\in\mathbb{N}}$, there exists a nested \textsc{fas} $\left(\mathscr{\varepsilon}_n, \mathscr{A}_n\right)_{n\in\mathbb{N}}$ with $A_n\subset\mathscr{A}_n$ for every $\todon$ and $\bigcup_{\todon} \mathscr{A}_n=\bigcup_{\todon} A_n$.
\end{refi}
\begin{proof}[Construction]
Consider, for every $\todon$, the $\varepsilon_n$ approximation $\mathscr{A}_n=\bigcup_{i=1}^n A_i$. Adding more points to the approximation $A_n$ does not change the fact that it is an $\varepsilon_n$ approximation (it might be a $\varepsilon'_n$ approximation with $\varepsilon'_n<\varepsilon_n$, but we are not interested in this). The sequence is still adjusted (with the same sequence of $\varepsilon_n$), since for these approximations, the sequence of $\gamma_n$ change to $\gamma'_n$ with $\gamma'_n\leqslant\gamma_n$ for all $\todon$, and hence \[\varepsilon_{n+1}<\frac{\varepsilon_n-\gamma'_n}{2}<\frac{\varepsilon_n-\gamma_n}{2}.\]
\end{proof}

So, we can work always with nested $\textsc{fas}$. The importance of nested $\textsc{fas}$ lies in the fact that they provide better injectivity for the map $\varphi$.
\begin{prop}\label{prop:injective}
Let $X$ be a compact metric space with a nested \textsc{fas} $\left(\varepsilon_n, A_n\right)_{n\in\mathbb{N}}$. Then $\varphi$ is injective on the set \[\varphi^{-1}\left(\bigcup_{\todon}A_n\right).\]
\end{prop}
\begin{proof}
This follows immediately from Lemma \ref{lem:inyectiva}
\end{proof}

One important property of compact metric spaces is that they always have a countable dense subset. Since the reconstruction is by finite spaces and hence the union of all is countable, we would like to apply Lemma \ref{lem:inyectiva} in a dense subset of $X$ to obtain injectivity on it. Actually, the union of all approximations is a good candidate for that. 
\begin{prop}
For every \textsc{fas} $\left(\varepsilon_n, A_n\right)_{\todon}$ of a compact metric space $X$, the set $\bigcup_{\todon}A_n$ is always dense in $X$. 
\end{prop}
\begin{proof}
For each open set $U\subset X$ there exists $x\in U$ and $\varepsilon>0$ such that $\textrm{B}(x,\epsilon)\subset U$. Let us select $n_0\in\mathbb{N}$ such that $\textrm{B}(x,\epsilon_{n_0})\subset\textrm{B}(x,\epsilon)$. Then, for any $a\in A_{n_0}$ with $\textrm{d}(x,a)<\epsilon_{n_0}$, we have that $\textrm{B}(x,\epsilon)\cap A_{n_0}\neq\emptyset$.
\end{proof}

For the last refinement, we first show that we can find the approximations using only points of dense subsets.
\begin{lem}\label{lem:dense}
 For every dense subset $Y\subset X$ of a compact metric space, and every $\epsilon>0$, there exists a finite $\epsilon$-approximation $A\subset Y$. 
\end{lem}
\begin{proof}
Let us consider the covering $\lbrace\ball{x}{\frac{\varepsilon}{2}}:x\in X\rbrace$ and a finite subcovering $\lbrace\ball{x_1}{\frac{\epsilon}{2}},\ldots,\ball{x_k}{\frac{\epsilon}{2}}\rbrace$. Now we take $y_1,\ldots,y_k\in Y$ such that $\dist{x_i}{y_i}<\frac{\epsilon}{2}$ for every $i=1,\ldots,k$, so $\{y_1,\ldots,y_k\}$ is an $\epsilon$-approximation of $X$.
\end{proof}

\begin{refi}[Approximations in dense subsets]\label{refi:dense}
Let $X$ be a compact metric space and $Y\subset X$ a dense subset. There is a nested \textsc{fas} such that $\bigcup_{\todon}A_n\subseteq Y$. Moreover, if $Y$ is countable, there is a nested \textsc{fas} also satisfying $\bigcup_{\todon}A_n=Y$.
\end{refi}
\begin{proof}[Construction]
The first \textsc{fas} can be obtained just applying Lemma \ref{lem:dense} to every $\varepsilon_n$ and making it nested with Refinement \ref{refi:nesting}. For the second condition, let us write $Y=\{y_1,y_2,\ldots,y_n,\ldots\}$. We just have to be sure that every element of $Y$ is included in $A_n$ for some $\todon$. There are several ways of doing so, two of which we describe next. First, we can take as finite approximations $A_1=\{y_1\}\cup  A'_1$, with $A'_1\subset Y$ an $\varepsilon_1$ approximation of $X$ and, for every $n>1$, $A_n=\{y_n\}\cup A_{n-1}\cup A'_n$ with $A'_n\subset Y$ an $\varepsilon_n$ approximation of $X$. This is possible since $Y$ is dense in $X$ by Lemma \ref{lem:dense}. A second way of doing it, and more explicit, could be to choose, for every $\todon$, the numbers $$r(n)=\min\big{\lbrace} i\in\mathbb{N}:\{y_1,\ldots,y_i\} \enspace\textrm{is a}\enspace\epsilon_n\enspace\textrm{approximation of}\enspace X\big{\rbrace},$$ it is clear that $r(n+1)\geqslant r(n)$ and then we can write the approximations as $A_n=\{y_1,\ldots,y_{r(n)}\}$ for every $\todon$ and we are done.
\end{proof}

We can state  the main result for countable dense subsets: they can be homeomorphically reconstructed as a subspace of the inverse limits. 
\begin{teo}
	For every countable dense subset of a compact metric space, $Y\subset X$, there exists a \textsc{fas} of $X$ such that there is a dense subset of $\mathcal{X}^*$ which is homeomorphic to $Y$. 
\end{teo}
\begin{proof}
	By Refinement \ref{refi:dense}, it is easy to obtain a \textsc{fas} $\fas$ of $X$ such that $A_n\subset A_{n+1}$, for every $\todon$, and $\bigcup_{\todon}A_n=Y$. 
	
	If we restrict the map $\varphi:\mathcal{X}\rightarrow X\supset Y=\bigcup_{\todon}A_n$ to the set $\varphi^{-1}(Y)$, we obtain by Proposition \ref{prop:injective} that $$\varphi\mid_{\varphi^{-1}(Y)}:\varphi^{-1}(Y)\longrightarrow Y$$ is injective and hence a homeomorphism. So $\varphi^{-1}$ is the desired set. We have the inclusions $\varphi^{-1}(Y)\subset\mathcal{X}^*\subset\mathcal{X}$, by Lemma \ref{lem:inyectiva}. Now, to see that $\varphi^{-1}(Y)$ is dense in $\mathcal{X}^*$. Let $V$ be any open set of $\mathcal{X}^*$ and $C\in V$ any point of it, where $C=(C_1,C_2,\ldots,C_n,\ldots)$. By Lemma 1 in \cite{MMreconstruction}, we can choose an open neighborhood $W$ of the form $$C\in W=\left(2^{C_1}\times\ldots\times2^{C_m}\times U_{2\epsilon_{m+1}}(A_{m+1})\times\ldots\right)\cap\mathcal{X}\subset V,$$ and select any $c\in C_m$. Then, there exists an index $n_0$ such that $c^*=(\ldots,\overset{n_0}{\overset{\downarrow}{\{c\}}},\{c\},\ldots)$, because of Lemma \ref{lem:inyectiva}. So $c^*\in W\cap\varphi^{-1}(Y)\subset V\cap\varphi^{-1}(Y)$, which implies $\overline{\varphi^{-1}(Y)}=\mathcal{X}^*$.
\end{proof}
\begin{obs}
	The inclusion $\varphi^{-1}(Y)$ of last proposition is proper: recall Example \ref{ex:interval} where $Y=\bigcup_{\todon}A_n$, with $A_n=\left\lbrace\frac{k}{3^{2n-3}}:k=0,1,\ldots,3^{2n-3}\right\rbrace$ and, while $\frac{1}{2}^*$ is obviously an element of $\mathcal{X}^*$, it does not belong to $\varphi^{-1}(Y)$, since $\frac{1}{2}$ does not belong to any approximation $A_n$.
\end{obs}

Let us exhibit a \textsc{fas} of this type for the rationals in the unit interval.

\begin{ej} \label{ex:unnested} Let $I=[0,1]$ be the unit interval. We want a \textsc{fas} of $I$ such that the approximations $\cup_{n\in\mathbb{N}}A_n$ are the rational numbers in the unit interval $\mathbb{Q}\cap I$. One way of doing this could be using the approximations consisting of subdivisions of the unit interval using $n+1$ points $\frac{1}{n}$ apart, for each $\todon$. But, it turns out that there is no appropriate sequence of $\varepsilon_n$ to make it a \textsc{fas}. Instead, we pick the following not nested \textsc{fas}. Define \[A_n=\left\lbrace\frac{k}{n^{2n-1}}:k=0,1,\ldots,n^{2n-1}\right\rbrace\text{ and }\varepsilon_n=\frac{1}{n^{2n-1}},\] for every $\todon$. We have a well defined \textsc{fas} and it is a simple but satisfying exercise for the reader to verify it, with $\gamma_n=\frac{\varepsilon}{2}=\frac{1}{2n^{2n-1}}$ for every $\todon$. 
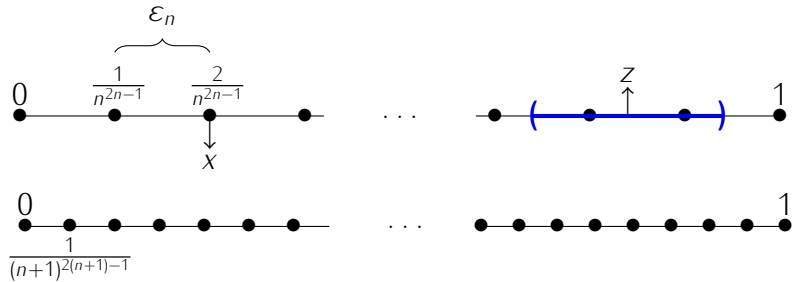
\begin{figure}[h!]
\captionsetup{width=.8\linewidth}
	\centering
	\begin{tikzpicture}[scale=10]
		\draw (0,0) -- (0.4,0);
		\draw (0.6,0) -- (1,0);
		\draw [above] node at (0,0){0};
		\draw [above] node at (1,0){1};
		\draw [above] node at (0.13,0){$\frac{1}{n^{2n-1}}$};
		\draw [above] node at (0.26,0){$\frac{2}{n^{2n-1}}$};
		\foreach \x in {0,1/8,2/8,3/8,5/8,6/8,7/8,1} {\draw node at (\x,0) {$\bullet$} ;}
		\draw [] node at (1/2,0){$\ldots$};
		\draw [decorate, decoration={brace, amplitude=5pt}, yshift=0.2em] (0.125,0) -- (0.25,0) 
		node [midway, above=0.5em] {$\varepsilon_n$};
		\draw [] node at (0.8,0.02){$\uparrow$};
		\draw [above] node at (0.8,0.03){$z$};
		\draw [line width=1pt,blue] node at (0.675,0){$($};
		\draw [line width=1pt,blue] node at (0.6763,0){$($};
		\draw [line width=1pt,blue] node at (0.925,0){$)$};
		\draw [line width=1pt,blue] node at (0.9235,0){$)$};
		\draw [line width=1.5pt, blue] (0.675,0) -- (0.925,0);
		
		\draw [] node at (0.25,-0.025){$\downarrow$};
		\draw [above] node at (0.25,-0.085){$x$};
		\draw [line width=1pt,blue] node at (0.675,0){$($};
		\draw [line width=1pt,blue] node at (0.6763,0){$($};
		\draw [line width=1pt,blue] node at (0.925,0){$)$};
		\draw [line width=1pt,blue] node at (0.9235,0){$)$};
		\draw [line width=1.5pt, blue] (0.675,0) -- (0.925,0);
	\end{tikzpicture}
	\\
	\begin{tikzpicture}[scale=10]
		\draw (0,0) -- (0.4,0);
		\draw (0.6,0) -- (1,0);
		\draw [above] node at (0,0){0};
		\draw [above] node at (1,0){1};
		\foreach \x in {0,0.05882352941,...,0.4} {\draw node at (\x,0) {$\bullet$} ;}
		\foreach \x in {0.6,0.65,...,1} {\draw node at (\x,0) {$\bullet$} ;}
		\draw [] node at (1/2,0){$\ldots$};
		\draw [below] node at (0.05882352941,0){$\frac{1}{(n+1)^{2(n+1)-1}}$};
	\end{tikzpicture}
	\caption{Two consecutive steps of the reconstruction of the unit interval, $A_n$ and $A_{n+1}$. Note the irregularity in the process, the \textsc{fas} is not nested. Here, we see $X^n=\{x\}$ but $Z^n$ has two points.}
	\label{fig:irregular}
\end{figure}
This way, we have obviously that $\bigcup_{n\in\mathbb{N}} A_n\subseteq\mathbb{Q}\cap I$. Actually, we have  $\bigcup_{n\in\mathbb{N}} A_n=\mathbb{Q}\cap I$, since for any rational number $x=\frac{a}{b}$ we have that for every step multiple of the denominator, $n=kb$ with $k\in\mathbb{N}$,
\[\frac{a}{b}=\frac{a}{b}\cdot\frac{n^{2n-1}}{n^{2n-1}}=\frac{a(kb)^{2n-1}}{bn^{2n-1}}=\frac{ak^{2n-1}b^{2n-2}}{n^{2n-1}}\in A_n.\]
That is, every rational in $I$ appears in an infinite number of approximations. By Refinement \ref{refi:nesting}, we can construct a nested \textsc{fas} $\{\varepsilon_n, \mathscr{A}_n\}_{\todon}$ of $X$ with $\bigcup_{\todon}\mathscr{A}_n=\bigcup_{\todon} A_n=\mathbb{Q}\cap I$. By Theorem \ref{teo:countable}, we have that $\varphi^{-1}(\mathbb{Q}\cap I)$ is a homeomorphic copy of $\mathbb{Q}\cap I$ into $\mathcal{X}^*$. Recall from Theorem 4 in \cite{MMreconstruction} that $\mathcal{X}^*$ is homeomorphic to the original space $I$, so this \textsc{fas} is a sort of completion of unit interval from rational  to real numbers.

Taking a deeper look into the original unnested \textsc{fas} $\fas$ in Figure \ref{fig:irregular}, we find that for every stage $n$, any point $x\in A_n$ (hence rational, $x=\frac{a}{b}$) satisfies $X^n=\{x\}$ and any point $z\notin A_n$ satisfies $X^n=\left\lbrace\frac{k}{n^{2n-1}}:k=j,j+1\right\rbrace$ for some $j=0,1,\ldots n^{2n-1}$. If $X^n$ is only one point, we have also $X_n^*=\{x\}$. Recall that every rational in the unit interval appears in an infinite number of indices of the \textsc{fas}, write them as $x_1, x_2,\ldots$. Hence, we have that $X_n^*=\left\lbrace \frac{a}{b}\right\rbrace$ for $n=x_1,x_2,\ldots$. For the rest of the indices, $X_n^*$ is determined by the bonding maps. That is, \[X^*=\left(\ldots,p_{x_1,x_1-1}\left(\left\lbrace\frac{a}{b}\right\rbrace\right),\overset{x_1}{{\overset{\downarrow}{\left\lbrace\frac{a}{b}\right\rbrace}}},\ldots,p_{x_2,x_2-1}\left(\left\lbrace\frac{a}{b}\right\rbrace\right),\overset{x_2}{{\overset{\downarrow}{\left\lbrace\frac{a}{b}\right\rbrace}}},\ldots\right)\] and hence, because of its maximality, it is the unique element of the inverse limit satisfying $\varphi(X^*)=x$. So, these points satisfy the same conditions than the points of Lemma \ref{lem:inyectiva} and hence we also have that $\varphi^{-1}(\mathbb{Q}\cap I)$ is a homeomorphic copy of $\mathbb{Q}\cap I$ into $\mathcal{X}^*$ with this unnested \textsc{fas}.
\end{ej}

From the previous example, and using the Sierpinski characterization theorem \cite{Ssur} of the rationals with the subspace topology from the reals as the only countable metric space without isolated points, we have the rationals has an inverse limit description in terms of finite $T_0$ spaces.
\begin{prop}
The rationals $\mathbb{Q}$ with the usual topology from $\mathbb{R}$ are a subspace of an inverse limit of an inverse sequence of finite $T_0$ spaces.
\end{prop}
\begin{proof}
By Sierpinski characterization theorem, $\mathbb{Q}$ is homeomorphic to $\mathbb{Q}\cap I$. By Example \ref{ex:unnested}, we have the desired description.
\end{proof}

\section{Topological reconstruction for countable and ultrametric spaces}\label{result}
\
Next, we consider countable compact metric spaces, compact metric spaces with a countable number of points, not to be confused with countably compact metric spaces. For them, we have a complete topological reconstruction by the inverse limit of a \textsc{fas}.
\begin{teo}\label{teo:countable}
	Let $X$ be a countable compact metric space. Then there exists a \textsc{fas} of $X$ such that the inverse limit $\mathcal{X}$ is homeomorphic to $X$.
\end{teo}
\begin{proof}
	We can write $X=\{x_1,x_2,\ldots,x_n,\dots\}$. By Theorem \ref{teo:countable} we find a \textsc{fas} $\fas$ for $X$ satisfying $x_n\in A_n$ and $A_n\subset A_{n+1}$ for every $\todon$. The first condition gives us $\bigcup_{\todon}A_n=X$ and the second one will make $\varphi$ injective on the set $$\varphi^{-1}\left(\bigcup_{\todon}A_n\right)=\varphi^{-1}(X)=\mathcal{X},$$ and then, $\varphi:\mathcal{X}\rightarrow X$ will be a homeomorphism.
\end{proof}

The second kind of compact metric spaces that are completely reconstructible by a \textsc{fas} are ultrametric spaces. An \emph{ultrametric} space $X$ is a metric space with an extra property of the distance. Instead of satisfying just the triangle inequality, they satisfy the \emph{strong triangle inequality}, that is:
$$\forall x,y,z\in X,\enspace\dist{x}{y}\leqslant\max\left\lbrace\dist{x}{z},\dist{y}{z}\right\rbrace.$$
This inequality gives us some properties that make the ultrametric spaces very special ones. For example, in ultrametric spaces, every triangle is isosceles, with the non equal side smaller than the other two. See Chapter 2 of \cite{Ra} for more properties and detailed proofs about ultrametric spaces. We shall need the following one: if two balls intersect, one is inside the other. 
\begin{prop}
For every two points $x,y$ of an ultrametric space $X$ and every $\epsilon\geqslant\delta>0$, if $\ball{x}{\epsilon}\cap\ball{y}{\delta}\neq\emptyset$, then $\ball{y}{\delta}\subseteq\ball{x}{\epsilon}$.
\end{prop}

We want to show that, for the case of compact ultrametric spaces, there exist some special \textsc{fas} whose limit recovers the topological type of the space. The key idea here is that, for those spaces, there are very special approximations.
\begin{lem}\label{lem:ultraaprox} Let $X$ be a compact ultrametric space. For every $\epsilon>0$, there exists a finite $\epsilon$-approximation of $X$, $A=\{x_1,\ldots,x_k\}$, such that $\ball{x_i}{\epsilon}\cap\ball{x_j}{\epsilon}=\emptyset$ for every $i\neq j$.
\end{lem}
\begin{proof}
	The covering by open balls $\{\ball{x}{\epsilon}:x\in X\}$ of $X$ has a finite subcover $\{\ball{x_1}{\epsilon},\ldots,\ball{x_k}{\varepsilon}\}$. So, $\{x_1,\ldots,x_k\}$ is an $\epsilon$ approximation of $X$. Now for any $i\neq j$ such that $\ball{x_i}{\epsilon}\cap\ball{x_j}{\epsilon}\neq\emptyset$ it turns out that $\ball{x_i}{\epsilon}=\ball{x_j}{\epsilon}.$
\end{proof}
We can state the reconstruction theorem about ultrametric spaces.
\begin{teo}[Ribota]
	Let $X$ be a compact ultrametric space. Then, there exists a \textsc{fas} such that its inverse limit $\mathcal{X}$ is homeomorphic to $X$. 
\end{teo}
\begin{proof}
	Let us consider any \textsc{fas} of $X$ in which, for every $\todon$, $A_n$ satisfies the property stated in the previous lemma. Consider, for every $\todon$, the so called \emph{nearby map} \mapeo{q_{A_n}}{X}{U_{2\varepsilon_n}(A_n)}{x}{\left\lbrace a\in A_n:\dist{a}{x}=\dist{A_n}{x}\right\rbrace.} This map is shown in \cite{MMreconstruction} to be well defined and continuous and it is easy to check that for every $C\in U_{2\varepsilon_{n+1}}(A_{n+1})$, \[p_{n,n+1}(C)=\bigcup_{c\in C}q_{A_n}(C).\]
	Then, for every $x\in X$ and every $\todon$, we have that $\card(q_{A_n}(x))=1$: Let us suppose that $a_1,a_2\in q_{A_n}(x)$. Then, $\dist{x}{a_1},\dist{x}{a_2}<\gamma_n<\epsilon_n$ but, in that case, we will have that $x\in\ball{a_1}{\epsilon_n}\cap\ball{a_2}{\epsilon_n}$ which is not possible. Then, $q_{A_n}:X\rightarrow A_n$ is actually a single valued continuous map. Moreover, if we restric to $A_{n+1}$, we obtain that $$q_{A_n}\mid_{A_{n+1}}=p_{n,n+1}\mid_{A_{n+1}}:A_{n+1}\longrightarrow A_n$$ is a continuous map. So, it makes sense to write the following diagram, 
	$$\xymatrix @C=15mm{X\ar[dr]^{q_{A_n}} \ar[d]_{q_{A_{n+1}}} & \\
		A_{n+1}\ar[r]_{p_{n+1,n}} & A_n,}$$ which, moreover, is commutative (compare with Proposition 2 of \cite{Mmultivalued}, where the same diagram is shown to be commutative, but up to homotopy type, for general compact metric spaces). If it would not be the case, then there would exist $a_1,a_2\in A_n$ with $q_{A_n}(x)=a_1$ and $p_{n,n+1}q_{A_{n+1}}(x)=a_2$. Clearly, $\dist{x}{a_1}<\epsilon_n$, but also 
	\begin{eqnarray*}\dist{x}{a_2} & \leqslant & \dist{x}{q_{A_{n+1}}(x)}+\dist{q_{A_{n+1}}(x)}{p_{n,n+1}q_{A_{n+1}}(x)}<\\ & < & \gamma_{n+1}+\gamma_n<\epsilon_{n+1}+\gamma_n<\frac{\varepsilon_n-\gamma_n}{2}+\gamma_n<\epsilon_n.
	\end{eqnarray*} 
	and this is imposible, since then $x\in\ball{a_1}{\epsilon_n}\cap\ball{a_2}{\epsilon_n}$. Adding that $q_{A_n}$ is always a surjective map distinguishing points of $X$ (see Corollary 3 on page 61 of \cite{MSshape}), we have that $X$ is the inverse limit $X=\lim_{\leftarrow}(A_n,p_{n,n+1})$. Now, it remains to see that every element of the inverse limit $C=(C_1,C_2,\ldots,C_n,\ldots)\in\mathcal{X}$ satisfies that $\card({C_n})=1$ for every $\todon$. If not, for any pair $a_1,a_2\in C_n$ we would have that $\dist{x}{a_1},\dist{x}{a_2}<\epsilon$, with $x=\lim_H\{C_n\}$, which, again, is not possible. So, we have that $$\mathcal{X}=\lim_{\leftarrow}\left(U_{2\epsilon_n}(A_n),p_{n,n+1}\right)_{n\in\mathbb{N}}=\lim_{\leftarrow}(A_n,p_{n,n+1})_{n\in\mathbb{N}}=X,$$ concluding the proof.
\end{proof}

We conclude with a question that may have been lingering in the reader's mind from the beginning.
\begin{preg} What kind of compact metric spaces are topologically reconstructible by a \textsc{fas}?
\end{preg}

\paragraph{Acknowledgments} The author is grateful to his thesis advisor, M.A. Morón, for his help regarding these results. 
\paragraph{Funding} This work has been partially supported by the research project PGC2018-098321-B-I00(MICINN). The author has been also supported by the FPI Grant BES-2010-033740 of the project MTM2009-07030 (MICINN).
	
\addcontentsline{toc}{chapter}{References}
\bibliographystyle{siam}
\bibliography{Countable}
	
	\vspace{1cm}
	\noindent\textsc{Diego Mondéjar}\\
	Departamento de Matemáticas\\
	CUNEF Universidad\\
	\texttt{diego.mondejar@cunef.edu}
	
\end{document}